\documentclass[letterpaper,article,oneside,11pt]{memoir}

\usepackage{amsmath}
\usepackage{amssymb}
\usepackage{amsthm}
\usepackage{mathrsfs}
\usepackage{lmodern}
\usepackage[T1]{fontenc}
\usepackage[utf8]{inputenc}
\usepackage{tikz}
\usepackage{xcolor}
\usepackage{enumitem}

\settrims{0pt}{0pt}
\setheadfoot{0.5in}{0.5in}
\setulmarginsandblock{1.0in}{1.0in}{*}
\setlrmarginsandblock{1.0in}{1.0in}{*}
\checkandfixthelayout

\usepackage[backend=biber,style=alphabetic,giveninits,url=false,doi=false,maxbibnames=50,maxcitenames=50,maxnames=50]{biblatex}
\defbibheading{bibliography}[Bibliography]{\section*{#1}\markboth{#1}{#1}}

\counterwithout{section}{chapter}
\counterwithout{equation}{chapter}
\numberwithin{equation}{section}

\renewcommand{\maketitle}{
 {\centering\textbf{\thetitle}\par
  \vspace{0.5em}
  \centering{\theauthor}\par 
  \vspace{2em}
 }
}

\setsecheadstyle{\bfseries}

\setsecnumdepth{subsubsection}
\setsecnumformat{\csname the#1\endcsname. }
\setsubsecheadstyle{\normalfont\bfseries}
\setaftersubsecskip{-0.5em}
\setsubsubsecheadstyle{\normalfont\bfseries}
\setaftersubsubsecskip{-0.5em}

\OnehalfSpacing
\newtheorem{theorem}[equation]{Theorem}

\newtheorem*{claim}{\textit{Claim}}
\newtheorem{proposition}[equation]{Proposition}

\newtheorem{lemma}[equation]{Lemma}
\newtheorem{fact}[equation]{Fact}

\theoremstyle{definition}

\newcommand{\R}{\mathbb{R}}
\newcommand{\N}{\mathbb{N}}
\newcommand{\Z}{\mathbb{Z}}
\newcommand{\SL}{\mathsf{SL}}

\newcommand{\nbar}{|\!|}
\newcommand{\intd}{\,\mathrm{d}}
\newcommand{\haar}{\mathsf{m}}

\newcommand{\define}[1]{\textbf{#1}}
\newcommand{\floor}[1]{\lfloor {#1} \rfloor}
\newcommand{\ceil}[1]{\lceil #1 \rceil}
\newcommand{\re}{\mathsf{Re}}
\newcommand{\im}{\mathsf{Im}}
\newcommand{\sphere}{\mathsf{S}}

\newcommand{\symdiff}{\triangle}
\newcommand{\ann}{\mathsf{Ann}}
\newcommand{\di}{\mathsf{a}}
\newcommand{\ro}{\mathsf{r}}
\newcommand{\sect}{\mathsf{Sec}}
\renewcommand{\arg}{\mathsf{arg}}
\newcommand{\hol}{\mathsf{h}}

\newcommand{\g}{\mathsf{g}}
\newcommand{\x}{\mathsf{x}}
\newcommand{\y}{\mathsf{y}}
\DeclareMathOperator{\D}{\mathsf{D}}

\begin{document}

\title{Uniform distribution of saddle connection lengths in all $\SL(2,\R)$ orbits}
\author{Donald Robertson, with an appendix by Benjamin Dozier}

\let\thefootnote\relax\footnote{Donald Robertson\\Department of Mathematics, University of Manchester, UK donald.robertson@manchester.ac.uk}

\let\thefootnote\relax\footnote{Benjamin Dozier\\Department of Mathematics, Cornell University, USA benjamin.dozier@cornell.edu}

\maketitle

\begin{abstract}
For every flat surface, almost every flat surface in its $\SL(2,\R)$ orbit has the following property: the sequence of its saddle connection lengths in non-decreasing order is uniformly distributed in the unit interval.
\end{abstract}

\section{Introduction}

Given a closed Riemann surface $X$ of genus at least two, every non-zero holomorphic one-form $\omega$ on $X$ has at least one zero.
If $\Sigma$ is the set of zeroes of such an $\omega$ then on $X \setminus \Sigma$ there is an induced atlas of charts to $\R^2$ all of whose transition maps are translations.
A \define{saddle connection} of a holomorphic one-form $\omega$ is any continuous map $v : [0,1] \to X$ such that $v^{-1}(\Sigma) = \{0,1\}$ and that $v|(0,1)$ is a geodesic segment in every chart of the induced atlas on $X \setminus \Sigma$.
Associated to each saddle connection $v$ is its holonomy vector
\[
\hol(v)
=
\left( \int\limits_v \re(\omega),\int\limits_v \im(\omega) \right)
\]
in $\R^2$.
The group $\SL(2,\R)$ is known to act on pairs $(X,\omega)$ via composition with the charts of the atlas.
The action takes saddle connections to saddle connections, and is such that $\hol(\g v) = \g \hol(v)$ for all $\g \in \SL(2,\R)$ and all saddle connections of $\omega$.

Enumerating the saddle connections of $\omega$ as $n \mapsto v_n$ in such a way that $n \mapsto \nbar \hol(v_n) \nbar$ is non-decreasing one can ask about the distribution of this sequence modulo one.
In \cite{MR4042164} it was proved for every $\SL(2,\R)$ invariant and $\SL(2,\R)$ ergodic probability measure $\mu$ on any stratum $\mathcal{H}$, that the sequence $n \mapsto v_n$ is uniformly distributed modulo one for $\mu$ almost every flat surface.
Here we prove the following refined result.

\begin{theorem}
\label{thm:mainTheorem}
For every flat surface $\omega_0$ of genus at least two and for almost every $\g \in \SL(2,\R)$, the sequence $n \mapsto v_n$ of lengths of saddle connections of $\g \omega_0$ listed in non-decreasing order is uniformly distributed modulo one.
\end{theorem}

This is an improvement over \cite[Theorem~1]{MR4042164} in two ways.
Firstly, the conclusion is stronger as the result holds for Haar almost-every flat surface in every $\SL(2,\R)$ orbit, not just for almost every flat surface with respect to any ergodic $\SL(2,\R)$ invariant probability measure on a stratum.
Secondly, the use of \cite{MR4031118} and the reliance on spectral gap results for the action of $\SL(2,\R)$ on a stratum found in the proof of \cite[Theorem~1]{MR4042164} are replaced with a more direct argument that only makes use of the following quadratic upper bound on the number of saddle connections in a sector extending \cite{MR3959357}.
The notation is explained after the statement of the theorem.

\begin{theorem}
\label{thm:secCount}
Fix a stratum $\mathcal{H}$ of flat surfaces of genus at least two.
There is a constant $c_4$ with the following property.
For every $\epsilon > 0$ and every flat surface $\omega \in \mathcal{H}$ there is $C(\omega,\epsilon) > 0$ such that
\begin{equation}
\label{eqn:arcBound}
R \ge \dfrac{C(\omega,\epsilon)}{|I|^{2+\epsilon}}
\Rightarrow
|\Lambda(\omega;R) \cap \sect(I)|
\le
c_4 |I| R^2
\end{equation}
for every arc $I \subset \sphere^1$.
Moreover, for each $\epsilon > 0$ the constant $C(\omega,\epsilon)$ depends continuously on $\omega$.
\end{theorem}

The explicit description in \eqref{eqn:arcBound} of how large $R$ must be in terms of the arc length of $I$ is of independent interest.
Theorem~\ref{thm:secCount} will be proved in Appendix~\ref{app:benProof}.

We now introduce some notation that will be used throughout the article.
Given an arc $I \subset \sphere^1$ write $\sect(I)$ for the subset of $\R^2$ that projects radially to $I$.
Given $\alpha < \beta$ with $\beta - \alpha < 2\pi$ write $\sect(\alpha,\beta)$ for the sector $\{ u \in \R^2 : \alpha \le \arg(u) < \beta \}$.
Given $0 \le A < B$ write $\ann(A,B)$ for the annulus $\{ u \in \R^2 : A \le \nbar u \nbar \le B \}$.
Fix a non-zero holomorphic one-form $\omega$ on a closed Riemann surface $X$ of genus at least two.
Write $\Lambda(\omega)$ for the set of saddle connections of $\omega$.
For $R > 0$ write $\Lambda(\omega;R)$ for the set of saddle connections of length at most $R$ on $\omega$.
We enumerate $\Lambda(\omega)$ by $n \mapsto v_n$ so that $n \mapsto \nbar \hol(v_n) \nbar$ is non-decreasing and $v \mapsto \arg(\hol(v)) \in [0,2\pi)$ is non-decreasing on each level set of $n \mapsto \nbar \hol(v_n) \nbar$.
Given $N \in \N$ write $\Xi(\omega;N)$ for $\{ v_n : n \le N \}$ where $n \mapsto v_n$ is the same enumeration as before.
Lastly, write $\ell(\omega;N) = \nbar \hol(v_N) \nbar$ for the length of the $N$th saddle connection on $\omega$ in terms of the above ordering, and $\ell(\omega) = \ell(\omega;1)$ for the length of the shortest saddle connection on $\omega$.
By an abuse of notation, given $B \subset \R^2$ we write $\Lambda(\omega) \cap B$ for the set of saddle connections of $\omega$ whose holomony vectors belong to $B$.
We will interpret $\Lambda(\omega;R) \cap B$ and $\Xi(\omega;N) \cap B$ similarly.

We thank Jon Chaika for suggesting this project and many useful conversations related to it.
We also thank the anonymous referee for a very thorough report.
Donald Robertson was supported by NSF grant DMS 1703597.

\section{Beginning the proof of Theorem~\ref{thm:mainTheorem}}
\label{sec:beginningProof}

Here we begin the proof of Theorem~\ref{thm:mainTheorem} by reducing it to Theorem~\ref{thm:toBound} via the Weyl criterion and a Borel-Cantelli argument.
Fix throughout this section a unit-area flat surface $\omega_0$ of genus at least two.
Masur's work~\cite{MR955824,MR1053805} yields constants $c_1,c_2 > 0$ with
\begin{equation}
\label{eqn:quadraticBounds}
(c_1eR)^2 \le |\Lambda(\omega;R)| \le \left( c_2 \frac{R}{e} \right)^2
\end{equation}
for all $R > 0$.
(The unorthodox expression is for later convenience.)

Write $\mu$ for Haar measure on $\SL(2,\R)$ and $\haar$ for Lebesgue measure on $\R$.
For each $t \in \R$ write
\[
\di^t = \begin{bmatrix} e^t & 0 \\ 0 & e^{-t} \end{bmatrix}
\qquad
\ro^t = \begin{bmatrix} \cos(t) & -\sin(t) \\ \sin(t) & \cos(t) \end{bmatrix}
\]
and define $\D(T) = \{ \ro^\theta \di^t \ro^\psi \in \SL(2,\R) : 0 \le \theta, \psi < 2 \pi, 0 \le t \le T \}$ for all $T > 0$.
We scale $\mu$ once and for all such that $\mu(\D(2)) = 1$.
Write $\chi_p(x) = \exp(2 \pi i p x)$ for each $p \in \Z$ and all $x \in \R$.

Our first reduction is to the following uniform distribution criterion.
Although we prove the theorem for all $\tau > 1$, it suffices for the proof of Theorem~\ref{thm:mainTheorem} to do so for a sequence of values $\tau > 1$ converging to 1.

\begin{theorem}
\label{thm:weylSum}
Fix $p \in \Z$ and $\tau > 1$.
One has
\begin{equation}
\label{eqn:criterionAgain}
\lim_{J \to \infty}
\frac{1}{\ceil{\tau^J}}
\sum_{v \in \Xi(\g \omega_0; \ceil{\tau^J})}
\chi_p(\nbar \hol(v) \nbar) = 0
\end{equation}
for $\mu$ almost-every $\g \in \D(1)$.
\end{theorem}
\begin{proof}[Proof of Theorem~\ref{thm:mainTheorem} assuming Theorem~\ref{thm:weylSum}]
Using Weyl's criterion for uniform distribution and \cite[Lemma~5]{MR4042164} it suffices for the proof of Theorem~\ref{thm:mainTheorem} to produce a sequence $\tau_1 > \tau_2 > \cdots \to 1$ with the property that
\begin{equation}
\label{eqn:criterion}
\lim_{J \to \infty} 
\frac{1}{\ceil{\tau_i^J}}
\sum_{v \in \Xi(\g \omega_0; \ceil{\tau_i^J})}
\chi_p(\nbar \hol(v) \nbar)
=
0
\end{equation}
holds for every $i,p \in \N$ and $\mu$ almost-every $\g \in \SL(2,\R)$.
Since the $\SL(2,\R)$ orbit of $\omega_0$ is countably covered by sets of the form $\D(1) \g \omega_0$ it suffices to work over $\D(1)$.
\end{proof}

Our next reduction will be via a Borel-Cantelli argument similar to the one in \cite{MR4042164}.
We give full details as there are several salient changes; most notably we average over $\Xi(\g \omega_0; N)$ and $\D(1)$ in place of $\Lambda(\omega; N^2)$ and a large compact subset of a stratum.
For each $N \in \N$ define
\[
f_N(\g) = \dfrac{1}{N} \sum_{v \in \Xi(\g \omega_0; N)} \chi_p(\| \hol(v) \|)
\]
for all $\g \in \SL(2,\R)$.

\begin{lemma}
\label{lem:doubling_thing}
For every $N \in \N$, every $0 < S \le 1$ and every $\sigma > 0$ the estimate
\begin{equation}
\label{eqn:before_knapp}
\mu(\{ \g \in \D(1) : |f_N(\g)|^2 \ge N^{-\sigma} \} )^2
\le
4 \int\limits_{\D(1)} \dfrac{\mu(\D(S) \x \cap \{ \g \in \D(1) : |f_N(\g)|^2 \ge N^{-\sigma} \})}{\mu(\D(S/2))} \intd \mu(\x)
\end{equation}
holds.
\end{lemma}
\begin{proof}
Fix $N \in \N$ and $0 < S \le 1$ and $\sigma > 0$.
Write $M = \{ \g \in \D(1) : |f_N(\g)|^2 \ge N^{-\sigma} \}$.
Consider the function $h : \SL(2,\R) \to [0,1]$ defined by
\[
h(\x) = \dfrac{\mu(\D(S/2) \x \cap M)}{\mu(\D(S/2))}
\]
for all $\x \in \SL(2,\R)$.

\begin{claim}
$\mu \left( \left\{ \x \in M : \dfrac{\mu(\D(S) \x \cap M)}{\mu(\D(S/2))} \ge \dfrac{\mu(M)}{2} \right\} \right)
\ge
\displaystyle \int\limits_{2h \ge \mu(M)} h \intd \mu$
\end{claim}
\begin{proof}
If we have $\y \in \SL(2,\R)$ with $h(\y) \ge \mu(M) / 2$ and $\x \in D(S/2) \y$ then
\[
\dfrac{\mu(\D(S) \x \cap M)}{\mu(\D(S/2))} \ge \dfrac{\mu(M)}{2}
\]
because in this case $\D(S) \x \cap M \supset \D(S/2) \y \cap M$.
Thus
\[
\bigcup_{2h(\y) \ge \mu(M)} \D(S/2) \y \cap M
\le
\left\{ \x \in M : \dfrac{\mu(\D(S) \x \cap M)}{\mu(\D(S/2))} \ge \dfrac{\mu(M)}{2} \right\}
\]
and, combined with
\begin{align*}
&{}
\mu \left( \bigcup_{2h(\y) \ge \mu(M)} \D(S/2) \y \cap M \right)
\\
={}
&
\dfrac{1}{\mu(\D(S/2))} \int\limits_{\SL(2,\R)} \mu \left( \D(S/2) \x \cap \bigcup_{2h(\y) \ge \mu(M)} \D(S/2) \y \cap M \right) \intd \mu(\x)
\\
\ge{}
&
\dfrac{1}{\mu(\D(S/2))} \int\limits_{2h \ge \mu(M)} \mu \left( \D(S/2) \x \cap \bigcup_{2h(\y) \ge \mu(M)} \D(S/2) \y \cap M \right) \intd \mu(\x)
\\
\ge{}
&
\dfrac{1}{\mu(\D(S/2))} \int\limits_{2h \ge \mu(M)} \mu \left( \D(S/2) \x \cap M \right) \intd \mu(\x)
\end{align*}
we have proved the claim.
\end{proof}
Since $M \subset \D(1)$ we have $h = 0$ outside $\D(2)$.
Thus
\[
\int\limits_{2h < \mu(M)} h \intd \mu \le \dfrac{\mu(M)}{2} \mu(\D(2))
\]
and we deduce
\[
\dfrac{\mu(M)}{2}
\le
\left( 1 - \dfrac{\mu(\D(2))}{2} \right) \mu(M)
\le
\mu \left( \left\{ \x \in M : \dfrac{\mu(\D(S) \x \cap M)}{\mu(\D(S/2))} \ge \dfrac{\mu(M)}{2} \right\} \right)
\]
from the claim and our scaling of $\mu$ because the integral of $h$ is $\mu(M)$ by Fubini.
An application of Markov's inequality then furnishes \eqref{eqn:before_knapp}.
\end{proof}

\begin{proposition}
\label{prop:new_diff}
There is $C_1 > 0$ such that for every $N \in \N$, every $0 < S \le 1$ and every $\sigma > 0$ we have
\begin{align*}
&
\mu\left( \left\{ \g \in \D(1) : |f_N(\g)|^2 > \dfrac{1}{N^\sigma} \right\} \right)
\\
\le{}
&
\left(
C_1 N^\sigma
\dfrac{1}{2 \pi} \int\limits_0^{2 \pi}
\int\limits_0^1
\dfrac{1}{2\pi} \int\limits_0^{2\pi}
\dfrac{1}{S} \int\limits_0^S
\dfrac{1}{N^2}
\sum_{v,w \in \Xi(\di^s \ro^\theta \di^t \ro^\phi \omega_0;N)}
\chi_p(\nbar \hol(v) \nbar)
\,
\overline{\chi_p(\nbar \hol(w) \nbar)}
\intd s
\intd \theta
\intd t
\intd \phi
\right)^{1/2}
\end{align*}
\end{proposition}
\begin{proof}
Fix $N \in \N$ and $0 < S \le 1$ and $\sigma > 0$.
Write $M = \{ \g \in \D(1) : |f_N(\g)|^2 \ge N^{-\sigma} \}$.
We have
\[
\mu(M)^2
\le
4 \int\limits_{\D(1)} \dfrac{\mu(\D(S) \x \cap M)}{\mu(\D(S/2))} \intd \mu(x)
=
4 \int\limits_{\D(1)} \dfrac{1}{\mu(\D(S/2))} \int\limits_{\D(S)} 1_M(\g \x) \intd \mu(\g) \intd \mu(\x)
\]
from Lemma~\ref{lem:doubling_thing}.
The definition of $M$ along with Markov's inequality gives
\[
\mu(M)^2
\le
4 N^\sigma
\int\limits_{\D(1)} \dfrac{1}{\mu(\D(S/2))} \int\limits_{\D(S)} |f_N( \g \x) |^2 \intd \mu(\g) \intd \mu(\x)
\]
and the right-hand integral becomes
\[
\dfrac{1}{2\pi} \int\limits_0^{2\pi} \dfrac{1}{2\pi} \int\limits_0^{2\pi} \int\limits_0^1
\dfrac{1}{\mu(\D(S/2))}
\dfrac{1}{2\pi} \int\limits_0^{2\pi} \dfrac{1}{2\pi} \int\limits_0^{2\pi} \int\limits_0^S
\sinh(t)
\sinh(s)
|f_N(
\ro^\varphi \di^s \ro^{\theta_1}
\ro^{\theta_2} \di^t \ro^\phi) |^2
\intd s \intd \varphi \intd \theta_1 \intd t  \intd \theta_2 \intd \phi
\]
upon writing $\mu$ in terms of the Cartan decomposition of $\SL(2,\R)$ as in \cite[Proposition~5.28]{MR1880691}.
The rotation $\ro^\varphi$ does not affect the sum and may be removed; the integrals over $\ro^{\theta_1}$ and $\ro^{\theta_2}$ together form a convolution and may be combined into a single term; we may bound $\sinh(t)$ by 2 and $\sinh(s)$ by $\sinh(S)$.
For some constant $C_1'$ one has
\[
\dfrac{\sinh(S)}{\mu(\D(S/2))} \sim \dfrac{C_1'}{S}
\]
and, together with the above, we have an absolute constant $C_1 > 0$ with
\[
\mu(M)^2
\le
C_1
N^\sigma
\dfrac{1}{2\pi} \int\limits_0^{2\pi}
\int\limits_0^1
\frac{1}{2\pi} \int\limits_0^{2\pi}
\frac{1}{S} \int\limits_0^{S}
\frac{1}{N^2}
\sum_{v,w \in \Xi(\di^s \ro^\theta \di^t \ro^\phi \omega_0;N)}
\chi_p(\nbar \hol(v) \nbar)
\,
\overline{\chi_p(\nbar \hol(w) \nbar)}
\intd s
\intd \theta
\intd t
\intd \phi
\]
as desired.
\end{proof}

From now on we fix the relationship
\[
S = \dfrac{1}{N^\delta}
\]
between $N \in \N$ and $S > 0$, where $\delta > 0$ is to be determined by future requirements.
(Ultimately $\delta = \tfrac{1}{3}$ will suffice.)
We now reduce the proof of Theorem~\ref{thm:mainTheorem} to the following statement.

\begin{theorem}
\label{thm:toBound}
There is $\eta > 0$ and $N_0 \in \N$ and a constant $C > 0$ such that
\begin{equation}
\label{eqn:toBound}
\dfrac{1}{2 \pi} \int\limits_0^{2 \pi}
\int\limits_0^1
\dfrac{1}{2\pi} \int\limits_0^{2\pi}
\dfrac{1}{S} \int\limits_0^S
\dfrac{1}{N^2}
\sum_{v,w \in \Xi(\di^s \ro^\theta \di^t \ro^\phi \omega_0;N)}
\chi_p(\nbar \hol(v) \nbar)
\,
\overline{\chi_p(\nbar \hol(w) \nbar)}
\intd s
\intd \theta
\intd t
\intd \phi
\le
\dfrac{C}{N^\eta}
\end{equation}
holds for all $N \ge N_0$.
\end{theorem}
\begin{proof}[Proof of Theorem~\ref{thm:mainTheorem} assuming Theorem~\ref{thm:toBound}]
Fix $p \in \Z$ and $\tau > 0$.
By Theorem~\ref{thm:weylSum} it suffices to verify \eqref{eqn:criterionAgain} for $\mu$ almost-every $\g \in \D(1)$.
Let $\eta > 0$ and $N_0 \in \N$ and $C > 0$ be as in the hypothesis.
Fix $0 < \sigma < \eta$.
Whenever $\ceil{\tau^J} \ge N_0$ and $\sigma > 0$ we have
\[
\mu
\left( \left\{
\g \in \D(1)
:
\left|
\frac{1}{\ceil{\tau^J}}
\sum_{v \in \Xi(\g \omega_0;\ceil{\tau^J})}
\chi_p(\nbar \hol(v) \nbar)
\right|^2
\ge
\frac{1}{\ceil{\tau^J}^\sigma}
\right\} \right)
\le
\left( C_1 C \dfrac{\ceil{\tau^J}^\sigma}{\ceil{\tau^J}^\eta} \right)^{1/2}
\]
by applying Proposition~\ref{prop:new_diff} and then \eqref{eqn:toBound}.
The right-hand side is summable over $J \in \N$ and the Borel--Cantelli lemma finishes the proof.
\end{proof}

There are two major steps in the proof of Theorem~\ref{thm:toBound}.
We outline them here and carry out the details in Sections~\ref{sec:annularEstimate} and \ref{sec:controllingPairs} respectively.
The steps will be combined to prove Theorem~\ref{thm:mainTheorem} in Section~\ref{sec:proof}.

\paragraph{Step 1: Annular estimate.}
We wish to move the action $\di^s$ inside the summation appearing in \eqref{eqn:toBound}.
This is not straightforward because $\Xi(\di^s \ro^\theta \di^t \ro^\phi \omega_0; N)$ and $\di^s \Xi(\ro^\theta \di^t \ro^\phi \omega_0;N)$ need not agree.
Indeed, if for $\ro^\theta \di^t \ro^\phi \omega_0$ one knows $v_N$ is close to the horizontal and $v_{N + 1}$ is close to the vertical then it may be that $\di^s v_{N + 1}$ is shorter than $\di^s v_N$.
To get around this issue it would suffice to find $\zeta > 0$ such that
\[
\left|
\Xi(\di^s \ro^\theta \di^t \ro^\phi \omega_0;N)
\symdiff
\di^s \Xi(\ro^\theta \di^t \ro^\phi \omega_0;N)
\right|
\ll
N^{1 - \zeta}
\]
holds for all $s,\theta,t,\phi$.
One can prove such an estimate using the effective count~\cite{MR4031118} for the number of saddle connections of length at most $R$ as $R \to \infty$ but our goal is to avoid the use of spectral gap results.
As a replacement we will find constants $\zeta > 0$ and $\lambda > 0$ such that
\[
\mu \left(
\left\{
0 \le t \le 1
:
\begin{gathered}
|
\Xi(\di^s \ro^\theta \di^t \ro^\phi \omega_0;N)
\symdiff
\di^s \Xi(\ro^\theta \di^t \ro^\phi \omega_0;N)
|
>
N^{1 - \zeta}
\\
\textup{ for some }
(s,\theta) \in [0,S] \times [0,2\pi)
\end{gathered}
\right\}
\right)
\ll
\dfrac{1}{N^\lambda}
\]
holds for all $0 \le \phi < 2\pi$.
We will do so in Section~\ref{sec:annularEstimate} using Theorem~\ref{thm:secCount}.

Although $\Xi(\ro^\theta \di^t \ro^\phi \omega_0; N)$ and $\ro^\theta \Xi(\di^t \ro^\phi \omega_0; N)$ may also disagree as sets of saddle connections (because we have decided to order saddle connections of the same length by increasing angle) in this case the summations over the two sets agree because the summands only depend on the lengths of the saddle connections.
It is therefore no problem, upon moving $\di^s$ inside in \eqref{eqn:toBound}, to move the action $\ro^\theta$ inside as well.

The purpose of the annular estimate is to reduce the verification of \eqref{eqn:toBound} to the production of some $\eta > 0$ such that
\begin{equation}
\label{eqn:toBound2}
\int\limits_0^1
\frac{1}{2\pi} \int\limits_0^{2\pi}
\frac{1}{S} \int\limits_0^{S}
\frac{1}{N^2}
\sum_{v,w \in \Xi(\di^t \ro^\phi \omega_0;N)}
\chi_p(\nbar \di^s \ro^\theta \hol(v) \nbar)
\,
\overline{\chi_p(\nbar \di^s \ro^\theta \hol(w) \nbar)}
\intd s
\intd \theta
\intd t
\ll
\frac{1}{N^\eta}
\end{equation}
 for every $0 \le \phi < 2 \pi$.

\paragraph{Step 2: Controlling pairs.}
To produce $\eta > 0$ such that \eqref{eqn:toBound2} holds, we apply a linearization to arrive at the quantity
\[
\int\limits_0^1
\frac{1}{2\pi} \int\limits_0^{2\pi}
\frac{1}{S} \int\limits_0^{S}
\frac{1}{N^2}
\sum_{v,w \in \Xi(\di^t \ro^\phi \omega_0;N)}
\chi_p(s \alpha(\ro^\theta \hol(v)))
\,
\overline{\chi_p(s \alpha(\ro^\theta \hol(w)))}
\intd s
\intd \theta
\intd t
\]
which we need to control for every $0 \le \phi < 2 \pi$.
From the proof of \cite[Lemma~12]{MR4042164} it suffice to bound
\[
\nbar \hol(v) \nbar | \sin(2 \theta_{v,w}) |
\]
from below by a power of $N$.
Here $\theta_{v,w}$ is the angle between the holonomy vectors of the saddle connections $v$ and $w$.
This issue will be dealt with in Section~\ref{sec:controllingPairs}.

\section{Annular estimate}
\label{sec:annularEstimate}

In this section we will establish the following theorem.

\begin{theorem}
\label{thm:annularEstimate}
There are $\lambda > 0$ and $\zeta > 0$ and $N_1 > 0$ such that
\[
\haar
\left(
\left\{ 0 \le t \le 1 :
\begin{gathered}
| \Xi(\di^s \ro^\theta \di^t \ro^\phi \omega_0; N) \triangle \di^s \Xi(\ro^\theta \di^t \ro^\phi \omega_0; N)| > N^{1 - \zeta} \\
\textup{ for some } (s,\theta) \in [0,S] \times [0,2\pi)
\end{gathered}
\right\}
\right)
\ll
\frac{1}{N^\lambda} 
\]
holds for all $N \ge N_1$ and all $0 \le \phi < 2\pi$.
\end{theorem}

The proof of Theorem~\ref{thm:annularEstimate} will take up the remainder of this section.
Recall that $\delta > 0$ defines $S = N^{-\delta}$.
Below, all requirements that $N$ be large enough depend only on $\omega_0$ and not on $\omega$.
Throughout this section fix $0 \le \phi < 2 \pi$ and write $\omega = \ro^\phi \omega_0$.

We begin with two lemmas that will be used to relate
\[
| \Xi(\di^s \ro^\theta \di^t \omega; N) \triangle \di^s \Xi(\ro^\theta \di^t \omega; N)| > N^{1 - \zeta}
\]
with counts for saddle connections in certain annuli.

\begin{lemma}
\label{lem:geodesicCountControl1}
For every $\g \in \SL(2,\R)$ one has
\[
\Xi(\di^s \g \omega;N) \subset \di^s \Lambda(\g \omega;e^{2s} \ell(\g \omega;N))
\]
for every $s > 0$ and every $N \in \N$.
\end{lemma}
\begin{proof}
If a saddle connection $v$ of $\g \omega$ has a length of more than $e^{2s} \ell(\g \omega;N)$ then $\di^s v$ has a length of more than $e^s \ell(\g \omega;N)$.
The saddle connection $\di^s v$ therefore cannot be amongst the first $N$ saddle connections of $\di^s \g \omega$ since $\di^s \Xi(\g \omega;N)$ has cardinality exactly $N$ and all of its members are saddle connections of $\di^s \g \omega$ of length at most $e^s \ell(\g \omega;N)$.
\end{proof}

\begin{lemma}
\label{lem:geodesicCountControl2}
For every $\g \in \SL(2,\R)$ one has
\[
\Xi(\di^s \g \omega; N)
\supset
\di^s \Lambda(\g \omega; e^{-2s} \ell(\g \omega;N))
\]
for every $s > 0$ and every $N \in \N$.
\end{lemma}
\begin{proof}
Every saddle connection in $\di^s \Lambda(\g \omega; e^{-2s} \ell(\g \omega;N))$ is a saddle connection of $\di^s \g \omega$ with length at most $e^{-s} \ell(\g \omega;N)$.
Moreover, no saddle connection of $\g \omega$ with length greater than $\ell(\g \omega;N)$ will, under the image of $\di^s$, be shorter than any saddle connection of $\di^s \Lambda(\g \omega;e^{-2s} \ell(\g \omega;N))$.
So all saddle connections in $\di^s \Lambda(\g \omega;e^{-2s} \ell(\g \omega;N))$ are amongst the first $N$ saddle connections of $\di^s \g \omega$.
\end{proof}

If the set
\begin{equation}
\label{eqn:ellipseFail}
\{ 0 \le t \le 1 : | \Xi(\di^s \ro^\theta \di^t \omega; N) \triangle \di^s \Xi(\ro^\theta \di^t \omega; N)| > N^{1 - \zeta} \textup{ for some } (s,\theta) \in [0,S] \times [0,2\pi) \}
\end{equation}
is empty for some $N \in \N$ then there is nothing to prove for that $N$.
Suppose that $t$ belongs to \eqref{eqn:ellipseFail} for some $N \in \N$.
We get $0 \le s \le S$ and $0 \le \theta < 2\pi$ depending on $t$ such that
\begin{equation}
\label{eqn:specificFailure}
\left|
\Xi(\di^s \ro^\theta \di^t \ro^\psi \omega; N) \symdiff \di^s \Xi(\ro^\theta \di^t \ro^\psi \omega; N)
\right|
\ge
N^{1 - \zeta}
\end{equation}
holds.
Given \eqref{eqn:specificFailure} we estimate
\begin{align*}
{}&
\left|
\Xi(\di^s \ro^\theta \di^t \omega; N) \symdiff \di^s \Xi(\ro^\theta \di^t \omega; N)
\right|
\\
\le
{}&
\left|
\di^s \Lambda(\ro^\theta \di^t \omega; e^{2s} \ell(\ro^\theta \di^t \omega; N)) \setminus \di^s \Xi(\ro^\theta \di^t \omega; N)
\right|
+
\left|
\di^s \Xi(\ro^\theta \di^t \omega; N) \setminus \di^s \Lambda(\ro^\theta \di^t \omega; e^{-2s} \ell(\ro^\theta \di^t \omega;N))
\right|
\\
\le
{}&
\left|
\Lambda(\ro^\theta \di^t \omega) \cap \ann(e^{-2S} \ell(\ro^\theta \di^t \omega; N) , e^{2S} \ell(\ro^\theta \di^t \omega; N))
\right|
\end{align*}
from Lemmas~\ref{lem:geodesicCountControl1} and \ref{lem:geodesicCountControl2}, where we have used $s \le S$ in deducing the last inequality.
Thus, whenever $t$ belongs to \eqref{eqn:ellipseFail} for some $N \in \N$, we have
\begin{align*}
N^{1 - \zeta}
&
\le
\left|
\Lambda(\ro^\theta \di^t \omega) \cap \ann(e^{-2S} \ell(\ro^\theta \di^t \omega; N) , e^{2S} \ell(\ro^\theta \di^t \omega; N))
\right|
\\
&
=
\left|
\Lambda(\di^t \omega) \cap \ann(e^{-2S} \ell(\di^t \omega; N) , e^{2S} \ell(\di^t \omega; N))
\right|
\end{align*}
because $\Lambda(\ro^\theta \di^t \omega) = \ro^\theta \Lambda(\di^t \omega)$ and $\ro^\theta$ does not change the length of the $N$th saddle connection.
We are led to consider the quantity
\begin{equation}
\label{eqn:annulusVisits}
E_N(v;t)
=
\begin{cases}
1 & \di^t v \in \ann(e^{-2S} \ell(\di^t \omega; N) , e^{2S} \ell(\di^t \omega; N)) \\
0 & \textup{otherwise}
\end{cases}
\end{equation}
defined for all $0 \le t \le 1$ and all $v \in \Lambda(\omega)$ and all $N \in \N$.
By its definition we therefore have
\[
N^{1-\zeta}
\le
\left|
\Lambda(\di^t \omega) \cap \ann(e^{-2S} \ell(\di^t \omega; N) , e^{2S} \ell(\di^t \omega; N))
\right|
\le
\sum_{v \in \Lambda(\omega)} E_N(v;t)
\]
whenever $t$ belongs to \eqref{eqn:ellipseFail}.
Thus
\[
\eqref{eqn:ellipseFail}
\subset
\left\{
0 \le t \le 1
:
\sum_{v \in \Lambda(\omega)} E_N(v;t)
\ge
N^{1 - \zeta}
\right\}
\]
which, together with Markov's inequality, gives
\begin{equation}
\label{eqn:desireAfterMarkov}
\haar(\eqref{eqn:ellipseFail})
\le
\frac{1}{N^{2(1-\zeta)}} \int\limits_0^1 \left| \sum_{v \in \Lambda(\omega)} E_N(v;t) \right|^2 \intd \haar(t)
\end{equation}
so our focus now is to bound
\begin{equation}
\label{eqn:afterMarkov}
\int\limits_0^1 \sum_{v \in \Lambda(\omega)} \sum_{w \in \Lambda(\omega)} E_N(v;t) E_N(w;t) \intd \haar(t)
\end{equation}
for $N$ large.
We begin with the following lemmas, which will allow us to restrict the sums in \eqref{eqn:afterMarkov} to thin annuli.

\begin{lemma}
\label{lem:nthLength}
There are constants $c_1,c_2 > 0$ such that for all $0 \le t \le 1$ we have
\[
\frac{1}{c_2} \frac{N^{0.5}}{\log N}
\le
\ell(\di^t \omega;N)
\le
\frac{1}{c_1} N^{0.5}
\]
for all $N \ge 2$.
\end{lemma}
\begin{proof}
First, note that
\[
|\Lambda(\di^t \omega) \cap \mathsf{B}(0,R)|
=
|\di^t \Lambda(\omega) \cap \mathsf{B}(0,R)|
=
|\Lambda(\omega) \cap \di^{-t} \mathsf{B}(0,R)|
\]
and $\mathsf{B}(0,\frac{1}{e}R) \subset \di^{-t} \mathsf{B}(0,R) \subset \mathsf{B}(0,eR)$ give from \eqref{eqn:quadraticBounds} the bounds $(c_1 R)^2 \le |\Lambda(\di^t \omega;R)| \le (c_2 R)^2$ for all $0 \le t \le 1$ and all $R > 0$.
Taking $R = N^{0.5}/c_1$ gives $N \le |\Lambda(\di^t \omega; N^{0.5}/c_1)|$ whence $\ell(\di^t \omega;N) \le N^{0.5}/c_1$.
Taking $R = N^{0.5}/c_2 \log N$ gives $N^{0.5}/ c_2\log N \le \ell(\di^t \omega;N)$.
\end{proof}

\begin{lemma}
\label{lem:bigAnnulus}
If $N^\delta \ge 2$ and $v \in \Lambda(\omega)$ is outside the annulus
\begin{equation}
\label{eqn:bigAnnulus}
\ann \left( \frac{1}{2ec_2} \frac{N^{0.5}}{\log N}, \frac{2e}{c_1} N^{0.5} \right)
\end{equation}
then $E_N(v;t) = 0$ for all $0 \le t \le 1$.
\end{lemma}
\begin{proof}
Fix $N \in \N$ with $N^\delta \ge 2$ and suppose $E_N(v;t) = 1$ for some $0 \le t \le 1$.
By Lemma~\ref{lem:nthLength} we have
\[
\frac{1}{e^{2S} c_2} \frac{N^{0.5}}{\log N} \le e^{-2S} \ell(\di^t \omega;N) \le e^{2S} \ell(\di^t \omega; N) \le \frac{e^{2S}}{c_1} N^{0.5}
\]
whence
\[
\di^t v \in \ann \left(\frac{1}{2 c_2} \frac{N^{0.5}}{\log N}, \frac{2}{c_1} N^{0.5} \right)
\]
because $e^{2S} \le 2$.
Therefore
\[
v \in \ann \left( \frac{1}{2ec_2} \frac{N^{0.5}}{\log N}, \frac{2e}{c_1} N^{0.5} \right)
\]
as $\di^t$ can lengthen vectors by a factor of at most $e$ and shorten them by a factor of at most $1/e$.
\end{proof}

Thus, for $N^\delta \ge 2$, only saddle connections of $\omega$ with holonomy inside \eqref{eqn:bigAnnulus} contribute to the summations in \eqref{eqn:afterMarkov}.
We continue by partitioning the annulus \eqref{eqn:bigAnnulus} into sectors as follows.
First define
\[
\overline{Z}
=
\sect(-\psi,\psi) \cup \sect(\tfrac{\pi}{2} - \psi, \tfrac{\pi}{2} + \psi) \cup \sect(\pi - \psi, \pi + \psi) \cup \sect(\tfrac{3\pi}{2} - \psi, \tfrac{3\pi}{2} + \psi)
\]
and
\[
Z = \overline{Z} \cap \Lambda \left( \omega; \frac{2e}{c_1} N^{0.5} \right)
\]
where $\psi = 2N^{-\gamma}$ for some $\gamma > 0$ yet to be determined.
(In fact $\gamma = \frac{1}{300}$ will suffice.)
Put
\begin{equation}
\label{eqn:kappa_def}
\kappa = \left\lfloor (\tfrac{\pi}{2} - 2 \psi) \ell(\omega;N)^\alpha \right\rfloor
\end{equation}
for some $\alpha > 0$ yet to be determined.
(In fact $\alpha = \frac{1}{100}$ will suffice.)
Decompose
\[
W
=
\Lambda(\omega) \cap \ann \left( \frac{1}{2ec_2} \frac{N^{0.5}}{\log N}, \frac{2e}{c_1} N^{0.5} \right)
\setminus
Z
\]
into $4\kappa$ subsets $W(1),\dots,W(4\kappa)$ each obtained by intersecting $W$ with annular arcs $\overline{W}(1),\dots,\overline{W}(4\kappa)$ of size $(\frac{\pi}{2} - 2\psi)/\kappa$.
Given $1 \le k \le 4\kappa$ let $\overline{V}(k)$ be the union of $\overline{W}(k)$, its reflections in the other quadrants, and any $\overline{W}(i)$ that are adjacent to any of these reflections.
(See Figure~\ref{fig:regions} for a schematic.)
Lastly, put $V(k) = \overline{V}(k) \cap \Lambda(\omega)$.
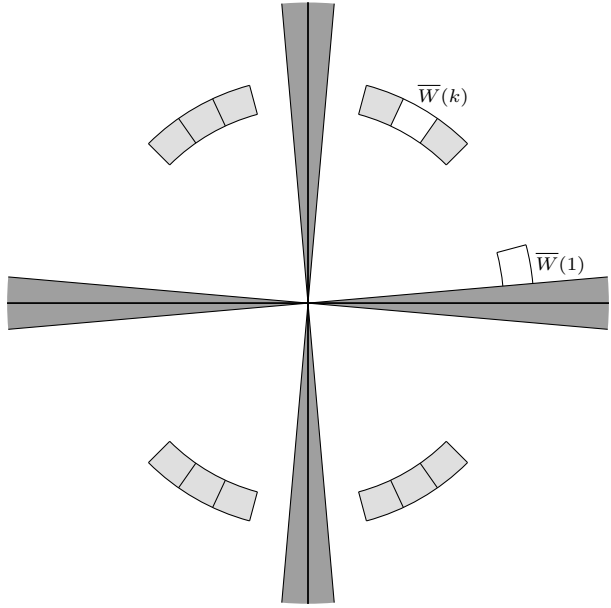
\begin{figure}[h]
\centering
\begin{tikzpicture}
\draw (-4,0) -- (4,0);
\draw (0,-4) -- (0,4);
\fill[gray!75] (0,0) -- (5:4) arc (5:-5:4) -- (0,0);
\fill[gray!75] (0,0) -- (85:4) arc (85:95:4) -- (0,0);
\fill[gray!75] (0,0) -- (175:4) arc (175:185:4) -- (0,0);
\fill[gray!75] (0,0) -- (265:4) arc (265:275:4) -- (0,0);
\draw (0,0) -- (5:4);
\draw (0,0) -- (85:4);
\draw (0,0) -- (95:4);
\draw (0,0) -- (175:4);
\draw (0,0) -- (185:4);
\draw (0,0) -- (265:4);
\draw (0,0) -- (275:4);
\draw (0,0) -- (355:4);
\fill[gray!25] (45:2.6) -- (45:3) arc (45:55:3) -- (55:2.6) arc (55:45:2.6);
\fill[gray!25] (65:2.6) -- (65:3) arc (65:75:3) -- (75:2.6) arc (75:65:2.6);
\draw (45:2.6) -- (45:3);
\draw (55:2.6) -- (55:3);
\draw (65:2.6) -- (65:3);
\draw (75:2.6) -- (75:3);
\draw (45:2.6) arc (45:75:2.6);
\draw (45:3) arc (45:75:3);
\node at (57:3.3) {\tiny{$\overline{W}(k)$}};
\draw (15:2.6) -- (15:3);
\draw (5:2.6) arc (5:15:2.6);
\draw (5:3) arc (5:15:3);
\node at (9:3.4) {\tiny{$\overline{W}(1)$}};
\begin{scope}[xscale=-1]
\fill[gray!25] (45:2.6) -- (45:3) arc (45:75:3) -- (75:2.6) arc (75:45:2.6);
\draw (-4,0) -- (4,0);
\draw (0,-4) -- (0,4);
\draw (45:2.6) -- (45:3);
\draw (55:2.6) -- (55:3);
\draw (65:2.6) -- (65:3);
\draw (75:2.6) -- (75:3);
\draw (45:2.6) arc (45:75:2.6);
\draw (45:3) arc (45:75:3);
\end{scope}
\begin{scope}[yscale=-1]
\fill[gray!25] (45:2.6) -- (45:3) arc (45:75:3) -- (75:2.6) arc (75:45:2.6);
\draw (-4,0) -- (4,0);
\draw (0,-4) -- (0,4);
\draw (45:2.6) -- (45:3);
\draw (55:2.6) -- (55:3);
\draw (65:2.6) -- (65:3);
\draw (75:2.6) -- (75:3);
\draw (45:2.6) arc (45:75:2.6);
\draw (45:3) arc (45:75:3);
\end{scope}
\begin{scope}[yscale=-1,xscale=-1]
\fill[gray!25] (45:2.6) -- (45:3) arc (45:75:3) -- (75:2.6) arc (75:45:2.6);
\draw (-4,0) -- (4,0);
\draw (0,-4) -- (0,4);
\draw (45:2.6) -- (45:3);
\draw (55:2.6) -- (55:3);
\draw (65:2.6) -- (65:3);
\draw (75:2.6) -- (75:3);
\draw (45:2.6) arc (45:75:2.6);
\draw (45:3) arc (45:75:3);
\end{scope}
\end{tikzpicture}
\caption{The annular sectors $\overline{W}(1)$ and $\overline{W}(k)$ for some $1 < k < \kappa$ are in white. The set $\overline{V}(k)$ consists of twelve annular sectors: the eleven light grey sectors together with $\overline{W}(k)$. The set $\overline{Z}$ is shown in dark grey. Note that $\overline{V}(1)$ only consists of eight regions as $\overline{Z}$ is not considered adjacent to any of our regions.}
\label{fig:regions}
\end{figure}

By Lemma~\ref{lem:bigAnnulus} the right-hand side of \eqref{eqn:afterMarkov} is bounded by the sum of the following expressions.
\begin{itemize}
\item
$\displaystyle{
\sum_{k=1}^{4\kappa}
\sum_{v \in W(k)}
\sum_{w \in V(k)}
\int\limits_0^1
E_N(v;t) E_N(w;t)
\intd t}$
\item
$\displaystyle{
\sum_{v \in Z}
\sum_{w \in \Lambda(\omega)}
\int\limits_0^1
E_N(v;t) E_N(w;t)
\intd t
+
\sum_{v \in \Lambda(\omega)}
\sum_{w \in Z}
\int\limits_0^1
E_N(v;t) E_N(w;t)
\intd t
}
$
\item
$
\displaystyle{
\sum_{k=1}^{4\kappa}
\sum_{v \in W(k)}
\sum_{\substack{w \in W\\ w \notin V(k)}}
\int\limits_0^1
E_N(v;t) E_N(w;t)
\intd t
}
$
\end{itemize}
and our goal now is to obtain power bounds for each of them.
This is carried out in the next two subsections: the first two expressions will be bounded via sectorial counts and the third via a separation argument.
Both make use of Theorem~\ref{thm:secCount}.

\subsection{Sectorial count}

Our goal here is to bound the sums
\begin{itemize}
\item
$\displaystyle{
\sum_{k=1}^{4\kappa}
\sum_{v \in W(k)}
\sum_{w \in V(k)}
\int\limits_0^1
E_N(v;t) E_N(w;t)
\intd t}$
\item
$\displaystyle{
\sum_{v \in Z}
\sum_{w \in \Lambda(\omega)}
\int\limits_0^1
E_N(v;t) E_N(w;t)
\intd t
}
$
\end{itemize}
by powers of $N$.

Let $C(\omega,\epsilon)$ be as in Theorem~\ref{thm:secCount}.
The parameter $\epsilon > 0$ will be determined later.
(In fact $\epsilon = \frac{1}{100}$ suffices.)
Fix $1 \le k \le 4 \kappa$.
With $I$ the appropriate sector of length $3(\frac{\pi}{2} - 2 \psi)/\kappa$ we have
\[
|\Lambda(\omega) \cap V(k)|
\le
4|\Lambda(\omega;\tfrac{2e}{c_1} N^{0.5}) \cap \sect(I)|
\le
4
\cdot
c_4
\cdot
\frac{3(\frac{\pi}{2} - 2\psi)}{\kappa}
\cdot
\left( \frac{2e}{c_1} N^{0.5} \right)^2
\]
whenever
\begin{equation}
\label{eqn:NlargeFromSector}
\frac{2e}{c_1} N^{0.5}
\ge
C(\omega,\epsilon)
\left( \frac{\kappa}{3(\frac{\pi}{2} - 2\psi)} \right)^{2+\epsilon}
\end{equation}
holds.

A priori, the occurrence of $\omega$ in \eqref{eqn:NlargeFromSector} means all subsequent statements requiring $N$ to be large enough will depend on $\omega$.
However, since $\epsilon > 0$ will be fixed and $\omega \mapsto C(\omega,\epsilon)$ is then continuous, the relation $\omega = \ro^\psi \omega_0$ implies that the apparent dependence on $\omega$ is in fact only a dependence on $\omega_0$.

Applying Lemma~\ref{lem:nthLength} and $\frac{1}{2} x \le \floor{x} \le x$ on $[1,\infty)$ to \eqref{eqn:kappa_def}, we bound $\kappa$ by
\[
\frac{\pi}{8 c_2^\alpha} \frac{N^{\frac{\alpha}{2}}}{(\log N)^\alpha}
\le
\frac{1}{2} \left( \frac{\pi}{2} - 2 \psi \right) \left( \frac{1}{c_2} \frac{N^{0.5}}{\log N} \right)^\alpha
\le
\kappa
\le
\left( \frac{\pi}{2} - 2 \psi \right) \left( \frac{1}{c_1} N^{0.5} \right)^\alpha
\le
\frac{\pi}{2 c_1^\alpha} N^{\frac{\alpha}{2}}
\]
whenever $N^\gamma \ge 16/\pi$.
Thus there is an absolute constant $c_6 > 0$ such that
\[
|\Lambda(\omega) \cap V(k)|
\le
4c_4 \cdot \frac{3 \pi}{2} \frac{8 c_2^\alpha}{\pi} \frac{(\log N)^\alpha}{N^\frac{\alpha}{2}} \cdot \frac{4e^2}{c_1^2} N
\le
c_6N^{1 - \frac{\alpha}{2}} (\log N)^{\alpha}
\le
c_6 N^{1 - \frac{\alpha}{4}}
\]
whenever
\[
N^{0.5} \ge \frac{c_1}{2e}
\cdot
\left( \frac{2}{3\pi} \right)^{2+\epsilon}
\cdot
\left( \frac{\pi}{8 c_2} \frac{N^{\frac{\alpha}{2}}}{(\log N)^\alpha} \right)^{2 + \epsilon}
\cdot
C(\omega,\epsilon)
\]
holds, which will be the case for $N$ large enough provided
\[
\alpha(2 + \epsilon) < 1
\tag{R}
\]
is in place.
We can therefore say, using Lemma~\ref{lem:bigAnnulus} and \eqref{eqn:quadraticBounds} to bound the first two sums, and bounding the integral by 1, that
\begin{equation}
\label{eqn:closeBound}
\sum_{k=1}^{4\kappa} \sum_{v \in W(k)} \sum_{w \in V(k)} \int\limits_0^1 E_N(v;t) E_N(w;t) \intd t
\ll
N^{1 - \frac{\alpha}{4}} \sum_{k=1}^{4\kappa} \sum_{v \in W(k)} \int\limits_0^1 E_N(v;t) \intd \haar(t) 
\ll
N^{2 - \frac{\alpha}{4}}
\end{equation}
holds whenever $N$ is large enough.

For the second sum we again apply Theorem~\ref{thm:secCount}, this time with $\epsilon' > 0$ to be determined ($\epsilon' = \frac{1}{100}$ suffices) and $I$ an appropriate arc of size $2\psi$ giving a constant $C(\omega,\epsilon')$ such that
\[
|\Lambda(\omega) \cap Z|
\le
4c_4 \cdot 4N^{-\gamma} \cdot \left( \frac{2e}{c_1} N^{0.5} \right)^2
\]
whenever
\[
\frac{2e}{c_1} N^{0.5}
\ge
\left( \frac{N^\gamma}{4} \right)^{2 + \epsilon'} C(\omega,\epsilon')
\]
holds.
Thus
\[
|\Lambda(\omega) \cap Z| \ll N^{1 - \gamma}
\]
holds provided
\[
\frac{1}{2} > \gamma(2 + \epsilon')
\tag{R}
\]
is the case.
Using Lemma~\ref{lem:bigAnnulus} and \eqref{eqn:quadraticBounds} to bound the second sum, we can say that
\begin{equation}
\label{eqn:axesBound}
\sum_{v \in Z}
\sum_{w \in \Lambda(\omega)}
\int\limits_0^1
E_N(v;t) E_N(w;t)
\intd t
\ll
N^{2 - \gamma}
\end{equation}
holds for $N$ large enough.

\subsection{Separation}

In this subsection we control the sum
\[
\sum_{k=1}^{4\kappa}
\sum_{v \in W(k)}
\sum_{\substack{w \in W \\ w \notin V(k)}}
\int\limits_0^1
E_N(v;t) E_N(w;t)
\intd t
\]
by a power of $N$.

Fix $1 \le k \le 4\kappa$ and fix saddle connections $v,w$ of $\omega$ with $v \in W(k)$ and $w \in W \setminus V(k)$.
Since $E_N(u;t)$ is unchanged when $u$ is reflected in either the horizontal or the vertical axis, we may assume that $v$ and $w$ are in the first quadrant.
Write $\theta_u = \arg(u)$.
We assume that $\theta_v > \theta_w$ as the alternative involves identical arguments.
The following properties are immediate consequences of that assumption, $v \in W(k)$ and $w \in W \setminus V(k)$.
\begin{enumerate}[label=\textbf{S\arabic*}.,ref=\textbf{S\arabic*}]
\item
\label{s:annulus}
$\dfrac{1}{2ec_2} \dfrac{N^{0.5}}{\log N} \le \nbar v \nbar, \nbar w \nbar \le \dfrac{2e}{c_1} N^{0.5}$
\item
\label{s:sectors}
$\dfrac{2}{N^\gamma} \le \theta_w < \theta_v \le \dfrac{\pi}{2} - \dfrac{2}{N^\gamma}$
\item
\label{s:separation}
$\theta_v - \theta_w > \dfrac{1}{\ell(\omega;N)^\alpha}$
\end{enumerate}
Indeed \ref{s:annulus} follows from Lemma~\ref{lem:bigAnnulus}, \ref{s:sectors} follows from the definitions of $\psi$ and $W$, and \ref{s:separation} follows from \eqref{eqn:kappa_def} and $N$ being large enough.

With these properties to hand, our goal in this subsection is bounding the Lebesgue measure of the set
\begin{equation}
\label{eqn:badTimes}
\{ 0 \le t \le 1 : E_N(v;t) E_N(w;t) = 1 \}
\end{equation}
by a negative power of $N$.

\begin{lemma}
\label{lem:maintainAngle}
For all $t > 0$ the angle between $\di^t v$ and $\di^t w$ is at least $\frac{1}{e^{2t}} (\theta_v - \theta_w)$.
\end{lemma}
\begin{proof}
The Cauchy mean value theorem gives
\[
\arctan \left( \frac{1}{e^{2t}} \frac{v_2}{v_1} \right)
-
\arctan \left( \frac{1}{e^{2t}} \frac{w_2}{w_1} \right)
\ge
\frac{1}{e^{2t}} 
\left(
\arctan \left(\frac{v_2}{v_1} \right)
-
\arctan \left(\frac{w_2}{w_1} \right)
\right)
=
\frac{1}{e^{2t}} (\theta_v - \theta_w)
\]
as desired.
\end{proof}

We frequently use the estimates
\[
\frac{1}{e^s} \nbar u \nbar \le \nbar \di^s u \nbar \le e^s \nbar u \nbar
\qquad \qquad
\theta_{\di^s u} \le \frac{1}{e^{2s}} \theta_u
\]
which hold for all $s \ge 0$ and all $u$ in the first quadrant.

If $E_N(v;t) E_N(w;t) = 0$ for all $0 \le t \le 1$ there is no need for a bound as \eqref{eqn:badTimes} will have zero measure.
We therefore assume also that there is a time $0 \le r \le 1$ at which $E_N(v;r) E_N(w;r) = 1$.
The definition of $E_N$ and Lemma~\ref{lem:nthLength} give
\begin{equation}
\label{eqn:badTimeDistance}
\Big| \nbar \di^r v \nbar - \nbar \di^r w \nbar \Big|
\le
2 \ell(\di^t \omega;N) \sinh(2 S)
\le
\frac{8}{c_1} N^{0.5 - \delta}
\end{equation}
as $\sinh(S) \sim S$ as $N \to \infty$.
Also \ref{s:separation} and Lemmas~\ref{lem:nthLength}, \ref{lem:maintainAngle} give
\begin{equation}
\label{eqn:badTimeAngle}
\theta_{\di^r v} - \theta_{\di^r w} \ge \frac{1}{e^{2r}} (\theta_v - \theta_w) \ge \frac{c_1^\alpha}{e^2} N^{-\frac{\alpha}{2}}
\end{equation}
and our first goal is to deduce from these that there is horizontal and vertical separation of $\di^r v$ from $\di^r w$.
Write $\chi$ for the angular separation $\theta_{\di^r v} - \theta_{\di^r w}$ and $Q = | \nbar \di^r v \nbar - \nbar \di^r w \nbar |$ for the difference in length.

\begin{figure}
\centering
\begin{tikzpicture}[scale=1.5]
\draw (0,3.6) arc (90:50:3.6) -- (50:3) arc (50:90:3) -- (0,3.6);
\fill[gray!25] (3.6,0) arc (0:30:3.6) -- (30:3) arc (30:0:3) -- (3,0);
\draw (-0.1,0) -- (4,0);
\draw (0,-0.1) -- (0,4);
\draw (3,0) arc (0:90:3);
\draw (3.6,0) arc (0:90:3.6);
\fill (40:3.25) circle (0.05);
\fill (28:3.1) circle (0.05);
\draw[dashed] (40:3.25) ++(0,-1.4) to (40:3.25);
\draw[dashed] (40:3.25) ++(0,1.4) to (40:3.25);
\draw[dashed] (40:3.25) ++(-1.4,0) to (40:3.25);
\draw[dashed] (40:3.25) ++(1.4,0) to (40:3.25);
\draw (0,0) to (50:4.4);
\draw (0,0) to (30:4.4);
\node[anchor=north west] at (40:3.25) {$\di^r v$};
\node[anchor=west] at (28:3.1) {$\di^r w$};
\end{tikzpicture}
\caption{If $\di^r w$ is separated in angle from $\di^r v$ as in \eqref{eqn:badTimeAngle} but $| \nbar \di^r v \nbar - \nbar \di^r w \nbar |$ is not too large as in \eqref{eqn:badTimeDistance} (so that $\di^r w$ belongs to the grey region) then we can say something about the horizontal and vertical separation of $\di^r v$ and $\di^r w$.}
\label{fig:gettingSeparation}
\end{figure}
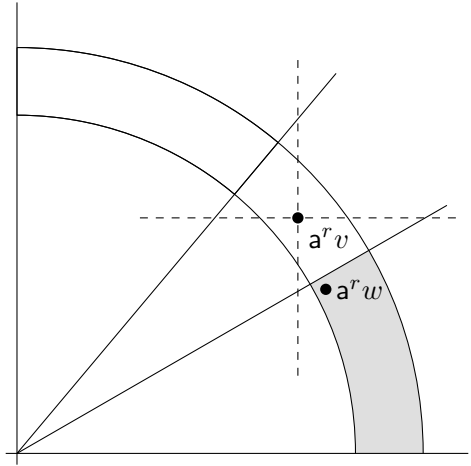

\begin{lemma}
\label{lem:separationIfClose}
There are constants $K > 0$ and $\xi > 0$ such that $(\di^r w)_1 - (\di^r v)_1 \ge K N^{0.5 - \xi}$ and $(\di^r v)_2 - (\di^r w)_2 \ge K N^{0.5 - \xi}$ for all $N$ large enough.
\end{lemma}
\begin{proof}
The coordinate $(\di^r w)_1$ cannot (cf.\ Figure~\ref{fig:gettingSeparation}) be smaller than $(\nbar \di^r v \nbar - Q) \cos(\theta_{\di^r v} - \chi)$ so from
\[
\cos(\theta_{\di^r v} - \chi) - \cos(\theta_{\di^r v})
=
\sin(\theta_{\di^r v}) \sin(\chi) - (1 - \cos(\chi)) \cos(\theta_{\di^r v})
\ge
\frac{\theta_{\di^r v}}{2} \frac{\chi}{2} - \frac{\chi^2}{2}
\]
we have by \eqref{eqn:badTimeDistance}, \eqref{eqn:badTimeAngle} and Lemma~\ref{lem:bigAnnulus} that
\begin{align*}
(\di^r w)_1 - (\di^r v)_1
&
\ge
\nbar \di^r v \nbar \frac{\chi}{2} \left( \frac{\theta_{\di^r v}}{2} - \chi \right) - Q
\\
&
\ge
\frac{1}{2e^2 c_2} \frac{N^{0.5}}{\log N}
\cdot
\frac{c_1^\alpha}{2e^2} N^{-\frac{\alpha}{2}}
\left(
\frac{1}{e^2} N^{-\gamma}
-
\frac{c_1^\alpha}{e^2} N^{-\frac{\alpha}{2}}
\right)
-
\frac{8}{c_1} N^{0.5 - \delta}
\end{align*}
because
\[
\dfrac{\theta_{\di^r v}}{2} - \chi
\ge
\frac{1}{2e^2} \theta_v - \chi
\ge
\frac{1}{e^2} \frac{1}{N^\gamma} - \frac{c_1^\alpha}{e^2} \frac{1}{N^{\frac{\alpha}{2}}}
>
0
\]
for $N$ large.
The above is contingent on the inequalities
\[
\gamma + \alpha < \delta
\qquad
\qquad
\gamma < \frac{\alpha}{2}
\tag{R}
\]
but provided they are satisfied we conclude there is a constant $K_1 > 0$ such that
\[
(\di^r w)_1 - (\di^r v)_1
\ge
\frac{K_1}{\log N} N^{0.5 - (\gamma + \alpha)}
\]
for all $N$ large enough.

Similarly, the largest $(\di^r w)_2$ can be is $(\nbar \di^r v \nbar + Q) \sin(\theta_{\di^r v} - \chi)$ so from
\[
\sin(\theta_{\di^r v}) - \sin(\theta_{\di^r v} - \chi)
=
(1 - \cos(\chi))\sin(\theta_{\di^r v}) + \cos(\theta_{\di^r v}) \sin(\chi)
\ge
\frac{\chi^2}{8} \frac{\theta_{\di^r v}}{2}
\]
we have
\begin{align*}
(\di^r v)_2 - (\di^r w)_2
&
\ge
\nbar \di^r v \nbar \frac{\chi^2}{8} \frac{\theta_{\di^r v}}{2} - Q
\\
&
\ge
\frac{1}{2e^2 c_2} \frac{N^{0.5}}{\log N}
\cdot
\frac{c_1^{2\alpha}}{8e^4} N^{-\alpha}
\cdot
\frac{1}{e^2} N^{-\gamma}
-
\frac{8}{c_1} N^{0.5 - \delta}
\end{align*}
as above.
Provided
\[
\gamma + \alpha < \delta
\tag{R}
\]
we conclude that there is a constant $K_2 > 0$ such that
\[
(\di^r v)_2 - (\di^r w)_2
\ge
\frac{K_2}{\log N} N^{0.5 - (\gamma + \alpha)}
\]
for all $N$ large enough.

To conclude take $K = \min \{ K_1, K_2 \}$ and $\xi = \gamma + 2 \alpha$ and $N_2 \ge \max \{ M_1, M_2 \}$ large enough.
\end{proof}

Define $f_{v,w}(t) = \nbar \di^t v \nbar - \nbar \di^t w \nbar$ for $t \in \R$.
Certainly $f_{v,w}$ is continuous.
Controlling the size of the set of those $0 \le t \le 1$ where $|f_{v,w}(t)|$ is small will be enough to bound the Lebesgue measure of \eqref{eqn:badTimes}.
From Lemma~\ref{lem:separationIfClose} and Facts~\ref{fac:uniqueZero}, \ref{fac:decreasing} in Appendix~\ref{app:calculus} we conclude that $f_{v,w}$ is decreasing and has a unique zero.
Since $f_{v,w}$ is continuous and decreasing, \eqref{eqn:badTimeDistance} furnishes $0 \le r_0 \le 1$ minimal with $|f_{v,w}(r_0)| \le \frac{8}{c_1} N^{0.5 - \delta}$.

\begin{lemma}
We have
\[
|f_{v,w}(r_0 + N^{-\varpi})| \ge \frac{8}{c_1} N^{0.5 - \delta}
\]
for some $\varpi > 0$ to be determined. (In fact, one can take $\varpi = \frac{1}{100}$.)
\end{lemma}
\begin{proof}
Suppose that the contrary holds.
Then
\[
\left| \nbar \di^{r_0 + t} v \nbar - \nbar \di^{r_0 + t} w \nbar \right|
=
|f_{v,w}(r_0 + t)|
\le
\frac{8}{c_1} N^{0.5 - \delta}
\]
for all $0 \le t \le N^{-\varpi}$.
We also have
\[
\frac{|\cos(2 \theta_{\di^{r_0 + t} v}) - \cos(2 \theta_{\di^{r_0 + t} w})|}{|2 \theta_{\di^{r_0 + t} v} - 2 \theta_{\di^{r_0 + t} w}|}
=
\sin(2 \xi)
\ge
\min \{ \sin(2 \theta_{\di^{r_0 + t} w}), \sin(2 \theta_{\di^{r_0 + t} v}) \}
\]
for some $\theta_{\di^{r_0 + t} w} < \xi < \theta_{\di^{r_0 + t} v}$.
For $N$ large enough $r_0 + t$ is at most $2$.
Therefore both angles are at least $\frac{2}{e^4} N^{-\gamma}$ and at most $\frac{\pi}{2} - \frac{2}{e^4} N^{-\gamma}$.
This implies
\[
\left|
\cos(2 \theta_{\di^{r_0 + t} v}) - \cos(2 \theta_{\di^{r_0 + t} w})
\right|
\ge
\frac{2}{e^4} N^{-\gamma} \cdot 2(\theta_{\di^{r_0 + t} v} - \theta_{\di^{r_0 + t} w})
\ge
\frac{2}{e^4} N^{-\gamma}
\cdot
\frac{2c_1^\alpha}{e^4} N^{-\frac{\alpha}{2}}
\]
after an application of Lemma~\ref{lem:maintainAngle}, \ref{s:separation}, Lemma~\ref{lem:nthLength}, and the fact that $r_0 +t \le 2$ for $N$ large enough.

In combination, for all $0 \le t \le N^{-\varpi}$ we get
\[
\Big| \nbar \di^{r_0 + t} v \nbar \cos(2 \theta_{\di^{r_0 + t} v}) - \nbar \di^{r_0 + t} w \nbar \cos(2 \theta_{\di^{r_0 + t} v}) \Big|
\le
\frac{8}{c_1} N^{0.5 - \delta}
\]
and by \ref{s:annulus} together with dilation control
\[
\nbar \di^{r_0 + t} w \nbar \Big| \cos(2 \theta_{\di^{r_0 + t} v}) - \cos(2 \theta_{\di^{r_0 + t} w}) \Big|
\ge
\frac{1}{2 e^3 c_2} \frac{N^{0.5}}{\log N}
\cdot
\frac{4c_1^\alpha}{e^8} N^{-\gamma - \frac{\alpha}{2}}
\]
so that \eqref{eqn:firstDerivative} gives
\[
|f_{v,w}'(r_0 + t)| \gg N^{0.5 - \gamma - \alpha}
\]
for all $0 \le t \le N^{-\varpi}$ provided
\[
\gamma + \alpha < \delta
\tag{R}
\]
holds.
But then the mean value theorem implies
\[
|f_{v,w}(r_0 + N^{-\varpi}) - f_{v,w}(r_0)|
\gg
N^{0.5 - \gamma - \alpha - \varpi}
\]
which, if one has
\[
0.5 - \gamma - \alpha - \varpi > 0.5 - \delta
\tag{R}
\]
implies $|f_{v,w}(r_0 + N^{-\varpi})| > \frac{8}{c_1} N^{0.5 - \delta}$ for $N$ large enough, giving the desired contradiction.
\end{proof}

Since $f_{v,w}$ is strictly decreasing the containment
\[
\{ 0 \le t \le 1 : E_N(v;t) E_N(w;t) = 1 \}
\subset
\left\{ 0 \le t \le 1 : |f_{v,w}(t)| \le \frac{8}{c_1} N^{0.5 - \delta}  \right\}
\subset
[r_0, r_0 + N^{-\varpi}]
\]
allows us to conclude that
\begin{equation}
\label{eqn:separationBound}
\sum_{k=1}^{4\kappa}
\sum_{v \in W(k)}
\sum_{\substack{w \in W \\ w \notin V(k)}}
\int\limits_0^1
E_N(v;t) E_N(w;t)
\intd t
\ll
N^{2-\varpi}
\end{equation}
by trivially bounding the sums using \eqref{eqn:quadraticBounds}.

\subsection{Proof of Theorem~\ref{thm:annularEstimate}}

If -- as they may be -- the parameters $\alpha, \gamma, \delta, \epsilon, \epsilon', \zeta, \varpi$ are chosen such that all of the requirements (R) above are satisfied then there is $N_1$ so large that \eqref{eqn:desireAfterMarkov}, \eqref{eqn:closeBound}, \eqref{eqn:axesBound}, \eqref{eqn:separationBound} together give
\[
\haar(\eqref{eqn:ellipseFail})
\le
\frac{1}{N^{2(1-\zeta)}} \int\limits_0^1 \left| \sum_{v \in \Lambda(\omega)} E_N(v;t) \right|^2 \intd \haar(t)
\ll
\frac{1}{N^{2(1-\zeta)}} \left( N^{2 - \frac{\alpha}{4}} + N^{2-\gamma} + N^{2-\varpi} \right)
\] and all $0 \le \phi < 2 \pi$.
for all $N \ge N_1$.
If
\[
\zeta < \min \left\{ \tfrac{1}{8} \alpha, \tfrac{1}{2} \gamma, \tfrac{1}{2} \varpi \right\}
\tag{R}
\]
is satisfied then we can find  $\lambda > 0$ satisfying the hypothesis.

\section{Controlling pairs}
\label{sec:controllingPairs}

In this section we wish to establish \eqref{eqn:toBound2} for all $0 \le \phi < 2 \pi$.
Throughout this section, fix $0 \le \phi < 2\pi$ and write $\omega = \ro^\phi \omega_0$.
For the quantity
\begin{equation}
\label{eqn:beforeLinearization}
\int\limits_0^1
\frac{1}{2\pi} \int\limits_0^{2\pi}
\frac{1}{S} \int\limits_0^{S}
\frac{1}{N^2}
\sum_{v,w \in \Xi(\di^t \omega;N)}
\chi_p(\nbar \di^s \ro^\theta \hol(v) \nbar)
\,
\overline{\chi_p(\nbar \di^s \ro^\theta \hol(w) \nbar)}
\intd s
\intd \theta
\intd t
\end{equation}
we wish to produce $\eta > 0$ such that
\begin{equation}
\label{eqn:beforeLinearizationGoal}
\eqref{eqn:beforeLinearization} \ll \dfrac{1}{N^\eta}
\end{equation}
for all $N$ large enough.

Writing
\[
\alpha(u) = \frac{u_1^2 - u_2^2}{\nbar u \nbar}
\qquad
\qquad
\beta(u) = \frac{2u_1u_2}{\nbar u \nbar}
\]
for any non-zero $u \in \R^2$, we begin by aplying the following linearization.

\begin{lemma}
\label{lem:linearization}
Whenever $\nbar u \nbar \le N^{0.5}$ we have
\[
\Big| \nbar u \nbar + \alpha(u) s - \nbar \di^s u \nbar \Big|
\ll
N^{0.5 - 2 \delta}
\]
for all $0 \le s \le S$.
\end{lemma}
\begin{proof}
Using \eqref{eqn:secondDeriv} the bound
\[
\left| \frac{f_u''(s)}{2} s^2 \right|
<
42 s^2 \nbar u \nbar
\ll
N^{0.5 - 2 \delta}
\]
gives the desired result via the Lagrange form of the remainder in Taylor's theorem.
\end{proof}

The requirement
\[
\delta > 0.25
\tag{R}
\]
is needed for Lemma~\ref{lem:linearization} to be useful.
When $\delta > 0.25$ it suffices for \eqref{eqn:beforeLinearizationGoal} to produce $\eta > 0$ such that the quantity
\begin{equation}
\label{eqn:afterLinearization}
\int\limits_0^1
\frac{1}{2\pi} \int\limits_0^{2\pi}
\frac{1}{S} \int\limits_0^{S}
\frac{1}{N^2}
\sum_{v,w \in \Xi(\di^t \omega;N)}
\chi_p(\nbar \ro^\theta \hol(v) \nbar + s \alpha(\ro^\theta \hol(v)) )
\,
\overline{\chi_p(\nbar \ro^\theta \hol(w) \nbar + s \alpha(\ro^\theta \hol(w)))}
\intd s
\intd \theta
\intd t
\end{equation}
satisfies
\begin{equation}
\label{eqn:afterLinearizationGoal}
\eqref{eqn:afterLinearization} \ll \dfrac{1}{N^\eta}
\end{equation}
for all $N$ large enough.

The reduction to \eqref{eqn:afterLinearizationGoal} follows from Lipshitz continuity of $\chi_p$.
Indeed, that gives
\[
| \chi_p(\nbar \ro^\theta \hol(v) \nbar + s \alpha(\ro^\theta \hol(v)) ) - \chi_p( \nbar \di^s \ro^\theta \hol(v) \nbar) |
\ll
\dfrac{1}{N^{2 \delta - 0.5}}
\]
for all $N$ large enough, whence
\begin{equation}
\label{eqn:reductionToLinearization}
|\eqref{eqn:beforeLinearization} - \eqref{eqn:afterLinearization}| \ll \dfrac{1}{N^{2 \delta - 0.5}}
\end{equation}
for all $N$ large enough.

We now work towards \eqref{eqn:afterLinearizationGoal}.
Fixing $0 \le t \le 1$ and fixing a parameter $\nu > 0$ to be chosen later ($\nu = \frac{1}{100}$ suffices) we may restrict the summation to those saddle connections $v,w$ for which $\nbar \hol(v) \nbar,\nbar \hol(w) \nbar \ge N^{0.5 - \nu}$ both hold.

We may also discard those $w$ for which the angle between $v$ and $w$ is at most $N^{-\frac{\alpha}{2}}$ by an application of Theorem~\ref{thm:secCount} similar to the one in Section~\ref{sec:proof}; if $C''(\omega,\epsilon'')$ is the attendant constant ($\epsilon'' = \frac{1}{100}$ suffices) then the requirement
\[
\frac{1}{2} > \frac{\alpha}{2}(\epsilon'' + 2)
\tag{R}
\]
will allow us to discard as desired.

Fixing $v,w$ satisfying both of these criterion, the quantity
\[
\frac{1}{2\pi} \int\limits_0^{2\pi}
\frac{1}{S} \int\limits_0^{S}
\chi_p(\nbar \ro^\theta \hol(v) \nbar + s \alpha(\ro^\theta \hol(v)) )
\,
\overline{\chi_p(\nbar \ro^\theta \hol(w) \nbar + s \alpha(\ro^\theta \hol(w)))}
\intd s
\intd \theta
\]
in absolute value is equal to
\[
\left|
\frac{1}{2\pi} \int\limits_0^{2\pi}
\frac{1}{S} \int\limits_0^{S}
\chi_p(s \alpha(\ro^\theta \hol(v)) )
\,
\overline{\chi_p(s \alpha(\ro^\theta \hol(w)))}
\intd s
\intd \theta
\right|
\]
because rotations do not change the lengths of holonomy vectors.

Define $A(v,w) \ge 0$ by $A(v,w)^2 = (\alpha(\hol(v)) - \nbar \hol(w) \nbar)^2 + \beta(\hol(v))^2$.
To proceed we quote the following estimate from \cite{MR4042164}.

\begin{lemma}
\label{lem:CRestimateApplication}
We have
\[
\left|
\frac{1}{2\pi} \int\limits_0^{2\pi}
\frac{1}{S} \int\limits_0^{S}
\chi_p(s \alpha(\ro^\theta \hol(v)) )
\,
\overline{\chi_p(s \alpha(\ro^\theta \hol(w)))}
\intd s
\intd \theta
\right|
\le
\frac{4}{\pi N^{0.5}}
+
\frac{N^\delta}{\pi A(v,w)} \left( \log N + \log \frac{\pi}{4} \right)
\]
for all $N$ large enough.
\end{lemma}
\begin{proof}
The estimate follows by duplicating (with $R = N^{0.5}$ and $\epsilon = 1$) the proof of \cite[Lemma~12]{MR4042164} up to \cite[Equation~(24)]{MR4042164} and the estimates immediately after \cite[Equation~(24)]{MR4042164}.
That much of the proof does not use the hypothesis.
\end{proof}

Our assumptions on $\| \hol(v) \|$, $\| \hol(w) \|$ and the angle between $v$ and $w$ give
\[
A(v,w) \ge \nbar \hol(v) \nbar | \sin(2(\theta_v - \theta_w)) | \ge N^{0.5 - \nu} \cdot c_1^\alpha N^{-\frac{\alpha}{2}}
\]
so that overall
\[
\left|
\frac{1}{2\pi} \int\limits_0^{2\pi}
\frac{1}{S} \int\limits_0^{S}
\chi_p(s \alpha(\ro^\theta \hol(v)) )
\,
\overline{\chi_p(s \alpha(\ro^\theta \hol(w)))}
\intd s
\intd \theta
\right|
\ll
N^{-0.5} + N^{\delta + \nu + \frac{\alpha}{2} - 0.5} \log N
\]
for our vectors $v,w$.
This is satisfactory provided
\[
\delta + \nu + \frac{\alpha}{2} < 0.5
\tag{R}
\]
holds, as we may then take $\eta = \delta + \nu + \frac{\alpha}{2} - 0.5$ to establish \eqref{eqn:afterLinearizationGoal}.

\section{Proof of main theorem}
\label{sec:proof}

In Section~\ref{sec:beginningProof} we reduced (via Theorem~\ref{thm:toBound}) the proof of Theorem~\ref{thm:mainTheorem} to the statement that
\begin{equation}
\label{eqn:toBound3}
\dfrac{1}{2\pi}
\int\limits_0^{2\pi}
\int\limits_0^1
\frac{1}{2\pi} \int\limits_0^{2\pi}
\frac{1}{S} \int\limits_0^{S}
\frac{1}{N^2}
\sum_{v,w \in \Xi(\di^s \ro^\theta \di^t \ro^\phi \omega_0;N)}
\chi_p(\nbar \hol(v) \nbar)
\,
\overline{\chi_p(\nbar \hol(w) \nbar)}
\intd s
\intd \theta
\intd t
\ll
\frac{1}{N^\eta}
\end{equation}
for some $\eta > 0$.
We finish here the proof of Theorem~\ref{thm:mainTheorem} by establishing \eqref{eqn:toBound3}.

\begin{lemma}
In order to prove \eqref{eqn:toBound3} it suffices to prove
\begin{equation}
\label{eqn:toBound4}
\int\limits_0^1
\frac{1}{2\pi} \int\limits_0^{2\pi}
\frac{1}{S} \int\limits_0^{S}
\frac{1}{N^2}
\sum_{v,w \in \Xi(\di^t \omega;N)}
\chi_p(\nbar \hol(\di^s \ro^\theta v) \nbar)
\,
\overline{\chi_p(\nbar \hol(\di^s \ro^\theta w) \nbar)}
\intd s
\intd \theta
\intd t
\ll
\frac{1}{N^\eta}
\end{equation}
for some $\eta > 0$.
\end{lemma}
\begin{proof}
Let $\Omega$ be the set \eqref{eqn:ellipseFail}.
The conclusion of Theorem~\ref{thm:annularEstimate} is that $\haar(\Omega) \ll N^{-\lambda}$.
When $t$ does not belong to $\Omega$ we have
\[
\left| \Xi(\di^s \ro^\theta \di^t \omega; N) \symdiff \di^s \Xi(\ro^\theta \di^t \omega;N) \right|
\le
N^{1 - \zeta}
\]
for all $ 0 \le s \le S$ and all $0 \le \theta < 2\pi$.
It follows that $N^{-\zeta} + N^{-\lambda}$ controls the difference between \eqref{eqn:toBound3} and \eqref{eqn:toBound4}.
\end{proof}

To apply the material of Section~\ref{sec:annularEstimate} -- which establishes Theorem~\ref{thm:annularEstimate} -- and the material of Section~\ref{sec:controllingPairs} -- which establishes \eqref{eqn:toBound4} -- we need to ensure that the requirements (R) above can be satisfied simultaneously. The choices
\[
\delta = \frac{1}{3}
\qquad
\alpha = \epsilon = \epsilon' = \epsilon'' = \nu = \varpi = \frac{1}{100}
\qquad
\gamma = \frac{1}{300}
\qquad
\zeta = \frac{1}{1000}
\]
show that this is possible, concluding the proof of Theorem~\ref{thm:mainTheorem}.

\appendix
\renewcommand{\thesection}{\Alph{section}} 

\section{Proof of Theorem \ref{thm:secCount} on counting in sectors by Benjamin Dozier}
\label{app:benProof}

Theorem~\ref{thm:secCount} is a more explicit version of \cite[Theorem 1.8]{MR3959357}.
Getting the explicit version is a matter of keeping careful track of the constants in various proofs in that paper.
The first step is to give explicit constants in \cite[Proposition 2.1, p.94]{MR3959357}.

\begin{proposition}
\label{prop:constants_improved}
Fix $\mathcal{H}$ and $0 < \delta < 1/2$.
Define $\alpha : \mathcal{H} \to \R$ by $\alpha(\omega) = 1/\ell(\omega)^{1+\delta}$.
There is a constant $b$ such that for any interval $I \subset \sphere^1$ there is a constant $c_I = O\left( \frac{1}{|I|^{1-2\delta}}\right)$ such that for any $\omega \in \mathcal{H}$
\[
\int\limits_I \dfrac{1}{\ell(\di^t \ro^\theta \omega)^{1+\delta}} \intd \theta
<
c_I \cdot e^{-(1-2\delta)T} \alpha(\omega) + b \cdot |I|
\]
for all $T \ge 0$.
\end{proposition}

The proof of Proposition~\ref{prop:constants_improved} is at the end of the appendix.
We first prove Theorem~\ref{thm:secCount} assuming Proposition~\ref{prop:constants_improved}.

\begin{proof}[Proof of Theorem~\ref{thm:secCount} assuming Proposition~\ref{prop:constants_improved}]
See~\cite[Theorem 1.8]{MR3959357}.
Explicit bounds for $c_I$ are not discussed there, but the argument in fact gives bounds as we now discuss.
From Proposition~\ref{prop:constants_improved}, we get that $c_{I'}=c_I = O \left(\frac{1}{|I|^{1-2\delta}}\right)$, where the implied constant depends on only on genus of the surface.
We will also use $\alpha(X) = 1/\ell(X)^{1+\delta}$.
For our lower bound on $R$, we can then take
\[
R_0=O \left( \frac{1}{|I|^{(2-2\delta)/(1-2\delta)}} \cdot \frac {1}{\ell(X)^{(1+\delta)/(1-2\delta)}} \right).
\]
Since we can choose $\delta$ as small as we wish (in particular, we can take $\delta = \epsilon/(2+2\epsilon)$), we get for every $\epsilon>0$,
\[R_0= \frac{C(X,\epsilon)}{|I|^{2+\epsilon}},
\]
where
\[
C(X,\epsilon) = O\left(\frac{1}{\ell(X)^{(1+(\epsilon/(2+2\epsilon))/(1-2 \epsilon/(2+2\epsilon))}}\right).
\]
We claim that we can take $C(X,\epsilon)$ to depend continuously on $X$.  It suffices to show that $\ell(X)$, the length of the shortest saddle connection, depends continuously on $X$.  We claim that $\ell$ equals the ``flat systole function'' $f$, which is defined as the length of the shortest curve or arc (starting/ending at zeros) that is not homotopic to a point (relative to zeros, in the case of an arc).  Clearly $f\le \ell$, since any saddle connection is such an arc.  To see that $\ell\le f$, note that any such curve/arc can be tightened to either (i) a union of saddle connections, or (ii) a closed geodesic, parallel copies of which form a cylinder.  In case (i), picking any one of the saddle connections in the union gives a saddle connection that has length at most that of the original curve/arc.  In case (ii), there is a saddle connection on the boundary of the cylinder that has length at most that of the original curve/arc.  Finally, the flat systole function clearly varies continuously, hence so does $\ell$.
\end{proof}

\begin{proof}[Proof of Proposition~\ref{prop:constants_improved}]
The result follows from the following modifications to \cite{MR3959357}.
\begin{itemize}
\item
\textbf{Explicit constants in \cite[Proposition 5.5, p.111]{MR3959357}:}  We can take $c_I = c_I^{(k)} = O_k\left( \frac{1}{|I|^{1-2\delta}}\right)$.

\item
\textbf{Addendum to \cite[Proof of Proposition 2.1 assuming Proposition 5.5, p.111]{MR3959357}:} Note that by the definition of $\alpha_i$, we have $\alpha_1(X) \ge \alpha_i(X)$, for each $i$ and every $X$ (this is because $\alpha_i(X)$ is defined in terms of the saddle connections on the boundary of a complex of complexity $i$; any such saddle connection forms a complex of complexity $1$, and thus is included in the definition of $\alpha_1$.)
  Thus we can replace the sum $\sum_{j\ge k}\alpha_j(X)$ that we get from \cite[Proposition 5.5]{MR3959357} with $M \cdot \alpha_1(X) = M/\ell(X)^{1+\delta}$, where $M$ is the complexity of $X$ (and then we can absorb the constant $M$ into the constant $c_I)$.

\item
\textbf{Explicit constants in \cite[Lemma 5.8, p.114]{MR3959357}:}  We can take $t_0 \approx \tau + \log\frac{1}{|I|}$.

\item
\textbf{Addendum to \cite[Proof of Proposition 5.5, p.116-118]{MR3959357}:}

\begin{itemize}
\item
We can take $m \approx t_0(\tau,|I|)/\tau \approx \frac{1}{\tau} \log\frac{1}{|I|}$.
\item
We can take $w_{\tau,I}^{(k)} = c_{I}^{(k+1)}c_2w_{\tau}$.
\item
We can take $c_{\tau,m} =\left( e^{-\tau(1-2\delta)}\right) ^{-m+1} + w_{\tau,I} \cdot e^{\tau(1-2\delta)} .$  From above, we have $w_{\tau,I}^{(k)} = c_{I}^{(k+1)}c_2w_{\tau}$.  For the maximal possible $k$, there are no terms from higher complexity, i.e. all the higher $\alpha_i$ are $0$, so, for this $k$, we get
    $$c_{\tau,I}^{(k)} = O_{k}\left( \left( e^{-\tau(1-2\delta)}\right)^{-m+1} \right) =O_k \left( e^{\tau(1-2\delta)\frac{1}{\tau}\log\frac{1}{|I|}} \right) = O_k \left( \frac{1}{|I|^{1-2\delta}} \right).$$
    Then inducting down by complexity, we get the same result for all $k$ (with different implied constant in the $O_k$).
    \qedhere
\end{itemize}
\end{itemize}
\end{proof}

\section{Length function}
\label{app:calculus}

In this appendix we collect various simple results about the function $f_v(t) = \nbar \di^t v \nbar$ defined on $\R$ for any vector $v \in \R^2$ with positive entries, and its relative $f_{v,w} = f_v - f_w$.
We have
\[
f_v(t) = \sqrt{e^{2t} v_1^2 + e^{-2t} v_2^2} = f_{\di^t v}(0)
\]
and define
\[
h_v(t) = e^{2t} v_1^2 - e^{-2t} v_2^2 = h_{\di^t v}(0)
\]
for convenience.
First note that
\[
f_v'(t) = \frac{h_v(t)}{f_v(t)}
\qquad
h_v'(t) = 2 f_v(t)^2
\]
from which
\begin{equation}
\label{eqn:secondDeriv}
f_v''(t)
=
\frac{2f_v(t)^4 - h_v(t)^2}{f_v(t)^3}
=
f_v(t) + \frac{4 v_1^2 v_2^2}{f_v(t)^3}
\end{equation}
follows.

Writing $V = \di^t v$ and $\theta_V$ for its argument we can write
\begin{equation}
\label{eqn:firstDerivative}
f_v'(t)
=
\frac{V_1^2 - V_2^2}{\nbar V \nbar} = \nbar V \nbar \Big( (\cos \theta_V)^2 - (\sin \theta_V)^2 \Big)
=
\nbar V \nbar \cos 2 \theta_V
\end{equation}
and
\[
f_v''(t)
=
\frac{2 (V_1^2 + V_2^2)^2 - (V_1^2 - V_2^2)^2}{\nbar V \nbar^3}
=
\frac{\nbar V \nbar^4 + 4 V_1^2 V_2^2}{\nbar V \nbar^3}
=
\nbar V \nbar + \frac{4V_1^2 V_2^2}{\nbar V \nbar^3}
=
\nbar V \nbar \Big (1 + (\sin 2 \theta_V)^2 \Big)
\]
for all $t \in \R$ from $2 V_1 V_2 = \nbar V \nbar^2 \sin 2 \theta_V$.
These calculations show $f_v$ is concave everywhere with a global minimum at
\[
m(v) = \frac{1}{2} \log \frac{v_2}{v_1}
\]
where $\theta_V = \frac{\pi}{4}$.

We now turn to $f_{v,w}$ assuming $v \ne w$.
We assume without loss of generality that $v$ and $w$ are in the first quadrant.
If $f_{v,w}$ has a zero it must be at
\begin{equation}
\label{eqn:zero}
r(v,w) = \frac{1}{4} \log \frac{v_2^2 - w_2^2}{w_1^2 - v_1^2}
\end{equation}
giving the following fact.

\begin{fact}
\label{fac:uniqueZero}
\leavevmode
\begin{itemize}
\item
If $w_1 > v_1$ and $v_2 > w_2$ the function $f_{v,w}$ has a unique zero.
\item
If $v_1 > w_1$ and $w_2 > v_2$ the function $f_{v,w}$ has a unique zero.
\item
In no other case does $f_{v,w}$ have a zero.
\end{itemize}
\end{fact}

\begin{fact}
\label{fac:decreasing}
If $f_{v,w}$ has a zero it is either strictly increasing or strictly decreasing. In other words, if $f_{v,w}$ has a zero then $f_{v,w}'$ does not.
\end{fact}
\begin{proof}
Switching the roles of $v$ and $w$ if necessary, we many assume $w_1 > v_1$ and $v_2 > w_2$.
We then have
\[
m(w)
=
\frac{1}{2} \log \frac{w_2}{w_1}
<
\frac{1}{2} \log \frac{v_2}{v_1}
=
m(v)
\]
and on the interval $(m(w),m(v)]$ the function $f_w$ is strictly increasing while on $[m(w),m(v))$ the function $f_v$ is strictly decreasing.
Thus $f_{v,w}$ is strictly decreasing on $[m(w),m(v)]$.

Next, note that since $v_2 > w_2$ we have $f_{v,w}(t) > 0$ for $t < r(v,w)$ and $f_{v,w}(t) < 0$ for $t > r(v,w)$.

Consider the case that $m(w) \le r(v,w)$.
Thus we have $\nbar W \nbar \le \nbar V \nbar$ and $\theta_V > \theta_W \ge \frac{\pi}{4}$ for all $t \le m(w)$.
Accordingly
\[
f_v'(t)
=
\nbar V \nbar \cos(2 \theta_V)
<
\nbar W \nbar \cos(2 \theta_W)
=
f_w'(t)
\]
and $f_{v,w}'$ is negative on $(-\infty,m(w)]$ whence $f_{v,w}$ is strictly decreasing on $(-\infty, m(v)]$.

In the case that $r(v,w) < m(w)$ then the above argument shows only that $f_{v,w}'$ is negative on $(-\infty,r(v,w)]$.
On the interval $(r(v,w),m(w))$ we have $\nbar W \nbar > \nbar V \nbar$ and $\theta_V > \theta_W > \frac{\pi}{4}$ giving $f_v'' < f_w''$ thereon.
Thus we may extend negativity of $f_{v,w}'$ to $(-\infty,m(w)]$ and $f_{v,w}$ is again strictly decreasing on all of $(-\infty,m(v)]$.

The cases $r(v,w) \le m(v)$ and $m(v) < r(v,w)$ are similar to the above, and altogether $f_{v,w}$ is strictly decreasing on all of $\R$.
\end{proof}

\printbibliography

\end{document}